\documentclass[english]{amsart}

\usepackage{amssymb}
\usepackage{amsthm}
\usepackage{amsmath}
\usepackage{psfrag}
\usepackage[all]{xy}
\usepackage{graphicx}
\usepackage{epstopdf}
\usepackage{color}
\usepackage{microtype}

\xyoption{arc}

\newcommand{\rr}{\mathbb{R}}
\newcommand{\zz}{\mathbb{Z}}

\newcommand{\hh}{\mathbb{H}}

\newcommand{\groupoid}[2]{\hspace*{-2.5pt} \xymatrix@C=16pt{ #1 \ar@<1.5pt>[r]^{\scriptscriptstyle  \ \ \alpha}\ar@<-1.5pt>[r]_{\scriptscriptstyle \ \, \beta} & #2} \hspace*{-2.5pt} }
\newcommand{\F}{\mathcal{F}} 
\newcommand{\A}{\mathcal{A}} 
\newcommand{\R}{\mathcal{R}}

\newcommand{\scirc}{{\scriptstyle \circ}}

\newcommand{\norminf}[1]{\|  #1  \| \raisebox{-0.6ex}{$_\infty$}}
\newcommand{\den}[1]{|  #1 |}
\newcommand{\poids}[2]{|  #1  |_{#2}}

\newcommand{\area}[1]{\mbox{area} \hspace{0.5pt} (#1)}
\newcommand{\vol}[1]{\mbox{vol}  \hspace{0.5pt}  (#1)}

\newcommand{\matrice}[4]{\left(
\begin{array}{cc}
 #1 & #2 \\
 #3 & #4
\end{array}
\right) }

\newtheorem{theorem}{Theorem}[section]
\newtheorem{corollary}[theorem]{Corollary}

\newtheorem{proposition}[theorem]{Proposition}

\theoremstyle{definition}
 \newtheorem{definition}[theorem]{Definition}
  
   \newtheorem{example}[theorem]{Example}
    
\newtheorem{remark}[theorem]{Remark}
\newtheorem{remarks}[theorem]{Remarks}

\begin{document}

\title[Averaging sequences]{Averaging sequences}

\author[F. Alcalde Cuesta]{Fernando Alcalde Cuesta} 
\address{Departamento de Xeometr\'{\i}a e Topolox\'{\i}a \\
Universidade de Santiago de Compostela \\  R\'ua Lope G\'omez de Marzoa s/n \\ 
15782 Santiago de Compostela, Spain}
\email{fernando.alcalde@usc.es}

\author[A. Rechtman]{Ana Rechtman}
\address{Institut de Recherche Math\'ematique Avanc\'ee\\
  Universit\'e de Strasbourg \\
  7 rue Ren\'e Descartes \\
  67084 Strasbourg, France}
\email{rechtman@unistra.fr}
\date{}

 \thanks{This work was partially supported by the Consejo Nacional de Ciencia y Tecnolog\'ia (CONACyT) in Mexico and the Xunta de Galicia INCITE09E2R207023ES in Spain.}
 


\begin{abstract}
In the spirit of Goodman-Plante average condition for the existence of
a transverse invariant measure for foliations, we give an averaging
condition to find tangentially smooth measures with prescribed
Radon-Nikodym cocycle. 
Harmonic measures are examples of tangentially smooth measures for
foliations and laminations. We also present sufficient hypothesis on the averaging condition under
which the tangentially smooth measure is harmonic.
\end{abstract}

\maketitle

\section{Introduction}
\label{introduccion}

Averaging sequences for foliations were introduced in the pioneering
work  of J. F. Plante \cite{plante} on the influence that the
existence of transverse invariant measures exerts on the structure of
a foliation. Although only the case of sub-exponential growth was
dealt with in \cite{plante}, Plante's approach is clearly reminiscent
of the classic work of E. F\o lner on  groups \cite{folner}. 
Using the same kind of ideas, S. E. Goodman and J. F. Plante exhibited
an averaging condition which guarantees the existence of transverse
invariant measures for foliations of compact manifolds \cite{goodmanplante}. 
\medbreak

In this paper we formulate a more general averaging condition which
gives rise to a tangentially smooth measure for a compact laminated space
$(M,\F)$. This condition may be
related to the $\eta$-F{\o}lner condition of \cite{alcalderechtman},
in the same spirit as F{\o}lner, but using a modified Riemannian
metric along the leaves. The modification is done by replacing any
complete Riemannian metric along the leaves with the product  of the metric with 
some density function. Namely, given a  compact laminated space and
a positive cocycle defined on the equivalence relation induced by 
the lamination on a total transversal, we prove that an $\eta$-F\o lner sequence gives rise to the existence of a tangentially smooth measure whose 
Radon-Nikodym cocycle is the given one. Moreover, we describe sufficient hypothesis for obtaining a harmonic measure. This is the content of Theorem~\ref{thharmonic}.
\medbreak

Before proving Theorem~\ref{thharmonic}, we start by analizing the discrete case. 
We define an averaging condition for 
any equivalence relation  $\R$ defined by a
finitely generated pseudogroup acting on a compact space and any
continuous cocycle $\delta : \R \to \rr^\ast_+$,  that we call
$\delta$-averaging condition.  In Theorem~\ref{discreteharmonic} we
prove that the existence of a $\delta$-averaging sequence gives a
quasi-invariant measure with Radon-Nikodym cocycle $\delta$.
Under some additional conditions, in particular if $\delta$ is harmonic,  the measure obtained is harmonic.  In this case, our result is reminiscent of Kaimanovich's characterization of amenable equivalence relations \cite{kaimcomptes}.
\medbreak

The paper is organized as follows. In Section~\ref{Spre} we review some preliminaries, in particular Section~\ref{Sgandp} contains the proof of Goodman and Plante's theorem.  The discussion of the discrete and continuous settings is splitted in two separate sections, Section~\ref{sectionaveragingdiscrete} and Section~\ref{sectionaveragingcontinuous}, respectively, which can be read independently. 
In Section~\ref{Sex} we analyze some explicit examples. The relation between the two types of averaging sequences will be briefly discussed in the final Section \ref{Sfc}.

\section{Preliminaries} \label{Spre}

\subsection{Laminations and equivalence relations}
\label{notaciondef}

A compact 
space $M$ admits a $d$-dimensional
lamination $\F$ of class $C^r$ with $1 \leq r \leq \infty$
if there exists a cover of $M$ by open sets $U_i$ homeomorphic to the
product  of an open disc $P_i$ in $\rr^d$ centered at the origin
and a locally compact separable metrizable space $T_i$. Thus, if
we denote the corresponding
foliated chart by $\varphi_i : U_i \to P_i \times T_i$, each 
$U_i$ splits into {\em plaques}  
$\varphi_i^{-1}(P_i \times \{y\})$. Each point $y \in T_i$ can also be identified with the point 
$\varphi_i^{-1}(0,y)$ in the {\em local transversal} $\varphi_i^{-1}(\{0\}\times T_i)$. 
In addition, the change of charts 
$\varphi_{j} \circ \varphi_i^{-1} : \varphi_i(U_i \cap U_j) \to \varphi_j(U_i \cap U_j)$
is given by
\begin{equation}
\label{ec:cambio_coor}
\varphi_{j} \circ \varphi_i^{-1}(x,y)
= (\varphi_{ij}^{y}(x),\gamma_{ij}(y))
\end{equation}
where $\gamma_{ij}$ is an homeomorphism between open subsets of $T_i$ and  $T_j$ and 
$\varphi_{ij}^y$ is a $C^r$-diffeomorphism depending continuously on $y$ in the $C^r$-topology. 
 We say that  $\A = \{ (U_i,\varphi_i) \} _{i\in I}$ is a {\em
  good foliated atlas}  if it satisfies the following conditions:

\begin{list}{\labelitemi}{\leftmargin=20pt}

\item[(i) \;] the cover $\mathcal{U} = \{U_i\}_{i\in I}$ is locally finite, hence finite;

\item[(ii) ] \label{boundedplaque}

each open set $U_i$ is a relatively compact subset of a foliated chart;

\item[(iii)] if $U_i \cap U_j\neq \emptyset$, there is a 
foliated chart containing $\overline{U_i \cap U_j}$, implying that each plaque of $U_i$ intersects at most one plaque of $U_j$.
\end{list}
\medbreak

\noindent
Each foliated chart $U_i$ admits a tangentially 
$C^r$-smooth Riemannian metric 
$g_i = \varphi_i^\ast g_0$ induced from a 
$C^r$-smooth Riemannian metric $g_0$ on $\rr^p$. 
We can glue together these local Riemannian metrics $g_i$ to obtain a global one $g$ using a tangentially $C^r$-smooth partition of unity.
From Lemma 2.6 of \cite{alcaldelozanomacho}, we know that any $C^r$
lamination of a compact space equipped with a $C^r$ foliated atlas
  $\A$ admits a  $C^\infty$ foliated atlas $C^r$-equivalent to
  $\A$.
\medskip

A discrete equivalence relation $\R$ is 
defined by $\F$ on the total transversal \mbox{$T = \sqcup T_i$:} the equivalence classes are the traces of the leaves on $T$. 
We can see $\R$ as the orbit equivalence relation defined by the
{\em  holonomy pseudogroup} $\Gamma$ of $\F$, generated by the local
diffeomorphisms $\gamma_{ij}$. These homeomorphisms form a finite
generating set, which we will denote $\Gamma^{(1)}$, that defines a 
{\em graphing} of $\R$. This means that each equivalence class $\R[y]$
is the set of vertices of a graph, and there is an edge joining two
vertices $z$ and $w$ if there is $\gamma \in \Gamma^{(1)}$ such that $\gamma(z)=w$.  We can define a graph metric 
$d_\Gamma(z,w)=\min \,  \{ \, n \, / \, \exists \gamma \in \Gamma^{(n)}  :  g(z)=w \, \}$, 
where $\Gamma^{(n)}$ are the elements that can be expressed as words of length at most $n$ in terms of $\Gamma^{(1)}$.  A {\em transverse invariant measure} for $\mathcal{F}$ is a measure on $T$ that is invariant under the action of $\Gamma$. It is quite rare for a measure of this kind to exist.

\begin{remark} \label{etale}
If  $\mathcal{F}$ has no holonomy ({\it i.e.} $\Gamma_y  = \{ \gamma \in \Gamma \, | \, \gamma(y) = y \}$ is trivial for all $y \in T$), we can endow $\R$ with the topology generated by the graphs of the elements of $\Gamma$. Then $\R$ becomes an {\em \'etale equivalence relation}, i.e.  the partial multiplication
$$((y,\gamma(y)),(\gamma(y),\gamma'(\gamma(y))) \in \R \ast \R \mapsto (y,\gamma' \scirc \gamma(y)) \in \R$$ and  the inversion $(y,\gamma(y)) \in \R \mapsto \! (\gamma(y),y) \in \R$
are continuous, and the left and right projections $\beta : (y,z) \in \R \mapsto y \in T$ and 
$\alpha : (y,z) \in \R \mapsto z \in T$ are local
homeomorphisms. In general, by considering the germs 
of the elements of $\Gamma$ at the points of their domains. we can replace $\R$ with the {\em transverse holonomy groupoid} \cite{haefliger1} that becomes similarly an {\em \'etale groupoid} \cite{renaultlnm}.
\end{remark}

\subsection{Compactly generated pseudogroups}

\noindent 
In the last section, we obtained a pseudogroup from a foliated
atlas. Here we will recall the {\em Haefliger equivalence} for
pseudogroups obtained from different atlases and its metric
counterpart in the compact case  that we will need later in Section
\ref{Sgandp}. For any compact laminated space $(M,\F)$ the holonomy pseudogroup
$\Gamma$ is {\em compactly generated} in the sense of
\cite{haefliger2}, meaning that:

\begin{list}{\labelitemi}{\leftmargin=18pt}

\item[(i) ] $T$ contains a relatively compact open set $T_1$ meeting all the orbits;

\item[(ii)] the reduced pseudogroup $\Gamma |_{T_1}$ (whose elements
have domain and range in $T_1$) admits a finite generating set $\Gamma^{(1)}$ (called a {\em compact generation system} of $\Gamma$ on $T_1$) so that each element $\gamma : A \to B$ of  $\Gamma^{(1)}$ is the restriction of an element $\overline{\gamma}$ of $\Gamma$ whose domain contains the closure of $A$. 
\end{list}

\noindent
Any probability measure $\nu_K$ on the compact set $K = \overline{T_1}$ that is preserved by the action of $\Gamma |_K$ extends to a unique Borel measure $\nu$ on $T$ which is $\Gamma$-invariant and finite on compact sets. We refer to Lemma 3.2 of \cite{plante}.
\medbreak

On the other hand, notice that $T$ is covered by the domains of a
family of elements of $\Gamma$ with range in $T_1$. The union of these
elements and their inverses defines the {\em fundamental equivalence}
between the holonomy pseudogroup $\Gamma$ and the reduced pseudogroup
$\Gamma |_{T_1}$. This is the base concept to define the {\em Haefliger
equivalence} of pseudogroups (see \cite{haefliger1} and \cite{haefliger2}):

\begin{definition}
Two pseudogroups $\Gamma_1$ and $\Gamma_2$ acting on
the spaces $T_1$ and $T_2$
are {\em Haefliger equivalent} if they are reductions of a same pseudogroup $\Gamma$
acting on the disjoint union $T = T_1 \sqcup T_1$, and both
  $T_1$ and $T_2$ meet all the orbits of $\Gamma$.
\end{definition}

The choice of generators for $\Gamma_1$ and $\Gamma_2$ defines a metric graph structure on the orbits, but the Haefliger equivalence between $\Gamma_1$ and $\Gamma_2$ may not preserve their quasi-isometry type. Let us recall this concept introduced by M. Gromov \cite{gromov}:

\begin{definition}
Two metric spaces $(M,d)$ and $(M',d')$ are {\em quasi-isometric} if
there exists a map $f:M \to M'$ and constants $\lambda \geq 1$ and $C \geq 0 $ such that 
$$\frac{1}{\lambda}d(y,z) - C \leq d'(f(y),f(z)) \leq \lambda d(y,z)+C$$
for all  $y,z \in M$ and $d'(y',f(M))\leq C$ for all $y' \in M'$.
\end{definition}

\begin{definition}
A Haefliger equivalence between two pseudogroups $\Gamma_1$ and $\Gamma_2$
acting on $T_1$ and $T_2$, respectively, is a {\em Kakutani
 equivalence} if $\Gamma_1$ and $\Gamma_2$ admit finite generating
systems such that their orbits, endowed with the graph metric, are quasi-isometric.
\end{definition}

According to Theorem 2.7 of \cite{lozanocayley} and Theorem 4.6 of  \cite{alvarezcandel}, if two
compactly generated pseudogroups $\Gamma_1$ and $\Gamma_2$ are
Haefliger equivalent, then there are compact generating systems on
$T_1$ and $T_2$, respectively, such that the pseudogroups become
Kakutani equivalent. These compact generating systems are called {\em  good} in \cite{lozanocayley} and {\em recurrent} in \cite{alvarezcandel}. The relevance of this, is that the existence of averaging sequences depends on the quasi-isometric type of the orbits (see \cite{alvarezcandel} and  \cite{kanai} for the details).

\subsection{Existence of transverse invariant measures}\label{Sgandp}

In this section we will discuss a sufficient condition for
the existence of a transverse invariant measure, which serves as motivation for Theorems \ref{discreteharmonic} and \ref{thharmonic}.
In \cite{goodmanplante}, Goodman and Plante formu\-late the following
proposition. Let us start with some definitions.

\begin{definition}
Let $A$ be a finite subset of $T$ and $\gamma$ an element of
$\Gamma$. We define the difference set
$$
\Delta_\gamma A = \{x\in T\;|\; x\in A, \gamma(x)\not\in A\}\cup \{x\in T
\;|\; x\not\in A, \gamma(x) \in A\},$$
with the convention that $\gamma(x)\not\in A$ holds if $\gamma(x)$ is not
defined. We denote the cardinality of $A$ by $|A|$.
\end{definition}

\begin{definition}
A sequence of finite  subsets $A_n$ of $T$ is an {\em averaging
  sequence} for $\Gamma$ if for all $\gamma \in \Gamma^{(1)}$ (and then for all $\gamma \in \Gamma$),
$$
\lim_{n \to \infty} \frac{| \Delta_\gamma A_n |}{|A_n|} = 0.
$$

\end{definition}

\begin{proposition}[Goodman-Plante \cite{goodmanplante}]
\label{averaging}
An averaging sequence $\{ A_n \}$ gives rise to a transverse invariant measure $\nu$ whose support  is contained in the limit set  $\lim_{n \to \infty}  A_n = 
 \{  y \in T  \,  |  \, \exists y_{n_k} \in A_{n_k} : y = \lim_{k \to \infty} y_{n_k} \}$.
\end{proposition}

The idea of the proof is the following. Assuming that $T$ is compact, we may construct a $\Gamma$-invariant probability measure on $T$ from the sequence of probability measures $\nu_n$ defined by 
$\nu_n(B) = | B \cap A_n | / |A_n|$
for every Borel set $B \subset T$. Accor\-ding to Riesz's representation theorem, each measure 
$\nu_n$ can be identified with a functional $I_n$ on the space $C(T)$
of continuous real-valued functions on $T$. The functionals $I_n$ are 
$$
I_n(f) = \frac{1}{|A_n|} \sum_{y \in A_n} f(y).
$$
By passing to a subsequence, if necessary, $I_n$ converges in the weak topology to a positive functional $I$ which determines a unique Borel regular measure $\nu$ such that
$I(f) = \int_T f d\nu$ for every $f \in C(T)$. The averaging condition implies that $I$ and $\nu$ are $\Gamma$-invariant since  for every $\gamma \in \Gamma$ and every 
$f \in C(T)$  with support on the range of $\gamma$, we have
$$
 | I (f \scirc \gamma) - I(f) | \leq \norminf{f} \lim_{n \to \infty} \frac{| \Delta_\gamma A_n |}{|A_n|} = 0.
$$ 
Finally, it is clear that $\nu(T) = 1$ and $supp(\nu) = \lim_{n \to \infty}  A_n$. 
\medbreak

In the non-compact case, by replacing $\Gamma$ and $\Gamma_1$ with
suitable reductions we can assume, without loss of generality, that the
fundamental equivalence between the holonomy pseudogroup $\Gamma$ and
its reduction $\Gamma_1$ to a relatively compact open subset $T_1$ of
$T$ becomes a Kakutani equivalence for some compact generation systems
on $T$ and $T_1$. Then, any averaging sequence $A_n$ for $\Gamma$ defines an averaging sequence $A_n \cap K$ for $\Gamma |_K$ where $K = \overline{T}_1$ is a compact subset of $T$.  Hence, we obtain a probability measure $\nu_K$ on $K$ that is invariant under $\Gamma |_K$. Now, we can extend $\nu_K$ to a unique Borel measure $\nu$ on $T$ which is $\Gamma$-invariant and finite on compact sets.

\begin{example} Consider a graph with bounded geometry, like any orbit $\Gamma(x)$ of the holonomy pseudogroup of a compact laminated space. This graph is said to be {\it F{\o}lner} if it contains a sequence of finite subsets of vertices $A_n$ such that
$|\partial A_n| / |A_n| \to 0$, where $\partial A_n$ denotes the boundary set with respect to the graph structure. Since $\Delta_\gamma A \subset  \partial A \cup \gamma^{-1}(\partial A)$ for any 
$\gamma\in \Gamma^{(1)}$,  we get that $|\Delta_\gamma A_n| \leq
2| \partial A_n|$, and we have an averaging sequence. In particular,
any orbit $\Gamma(x)$ having sub-exponential growth is an example of
F{\o}lner graph
since
$$
\liminf_{n \to \infty} \frac{| A_{n+1} - A_{n-1}|}{|A_n|}  = 0,
$$ 
where $A_n = \Gamma^{(n)}(x)$.
\end{example} 

Using the one-to-one correspondence between foliated cycles and transverse invariant measures stablished by D. Sullivan \cite{sullivancycles}, it is not difficult to show the follo\-wing continuous version of Goodman-Plante's result: 

\begin{proposition}[Goodman-Plante \cite{goodmanplante}]
\label{continuousaveraging}
Let $\{ V_n \}$ be an averaging sequence for $\F$, i.e. a sequence of compact domains $V_n$ (of dimension $d$) in the leaves such that 
$$
\lim_{n \to \infty} \frac{\area{\partial V_n}}{\vol{V_n}} = 0
$$
where $area$ denotes the $(d-1)$-volume and $vol$ the $d$-volume with
respect to the complete Riemannian metric along the leaves.  Then $\{ V_n \}$ gives rise to a transverse invariant measure $\nu$ whose support  is contained in the saturated  limit set 
$\lim_{n \to \infty}  V_n = \{ \,  p \in M  \, / \,  \exists p_{n_k}  \in V_{n_k}  : p = \lim_{k \to \infty} p_{n_k}   \, \}$.
\end{proposition}

\noindent
Recall that a {\em foliated $d$-form}  $\alpha \in \Omega^d(\F)$ is a family of differentiable $d$-forms over the plaques of $\A$ depending continuously on the transverse parameter and which agree on the intersection of each pair of  foliated charts. A {\em foliated $r$-cycle} is a continuous linear functional 
$\xi : \Omega^d(\F) \to \rr$ strictly positive on strictly positive forms and null on exact forms with respect to the leafwise exterior derivative $d_\F$. 
Any averaging sequence $V_n$ defines the sequence of foliated currents
$$\xi_n(\alpha)=\frac{1}{\vol{V_n}}\int_{V_n}\alpha$$
where $\alpha$ is a foliated $d$-form. By passing to a subsequence, if necessary, we have a 
limit current $\xi=\lim_{n\to\infty}\xi_n$. Since the boundaries of the domains $V_n$ vanish asymptotically, Stokes' theorem implies that $\xi$ is a foliated $d$-cycle
\cite{sullivancycles}. 

\setcounter{equation}{0}

\section{Averaging sequences in the discrete setting}
\label{sectionaveragingdiscrete}

The main objective of this section is to prove the existence of a
harmonic measure for an \'etale equivalence relation $\R$ that
contains a modified averaging sequence. Initially, we will assume that  $\R$
is given by a free action of a pseudogroup $\Gamma$ on a compact space
$T$, but some generalizations will be discussed later.
In Section \ref{quasiinvariant}, we will  define a weighted measure on the equivalence classes, that will allow us to recall the notion of {\em modified averaging sequence}
introduced by V. A. Kaimanovich in \cite{kaimcomptes} and \cite{kaimexamples}. 
Given a continuous cocycle $\delta : \R \to \rr^*_+$,
the {\em Radon-Nikodym problem} is to determine the set of 
probability measures $\nu$ on $T$ which are quasi-invariant and 
admit $\delta$ as their Radon-Nikodym derivative \cite{renault}.
Theorem \ref{discreteharmonic} gives a positive answer to this problem in
 the presence of a modified averaging sequence.
 
\subsection{Quasi-invariants measures} 
\label{quasiinvariant}

 Let $\nu$ be a quasi-invariant measure on $T$. As usual, we will assume that $\nu$ is a regular Borel measure that is finite on compact sets. Integrating the counting measures on the fibers of the left projection $\beta(y,z) = y$  with respect to $\nu$ gives the {\it left counting measure} $d\widetilde{\nu}(y,z)=d\nu(y)$. 
 Indeed, for each Borel set $A \subset \R$, we define
$\widetilde{\nu}(A) = \int \den{A^y} d\mu(y)$ where $\den{A^y}$ is the cardinal of the set  
$A^y = \{ z \in T / (y,z) \in A \} \subset \R [y]$.
The same is valid for the right projection 
$\alpha(y,z) = z$ and we get  the {\it right counting measure}
 $d\widetilde{\nu}^{-1}(y,z)=d\widetilde{\nu}(z,y)=d\nu(z)$.
Then $\widetilde{\nu}$ and $\widetilde{\nu}^{-1}$ are equivalent measures if and only if $\nu$ is quasi-invariant, in which case the Radon-Nikodym derivative is given by
$\delta(y,z)= d\widetilde{\nu}/d\widetilde{\nu}^{-1}(y,z)$.  We refer to \cite{mooreschochet}, \cite{kaimcomptes}, \cite{renaultlnm} and \cite{renault}.

\begin{definition} 
A {\em cocycle with values in $\rr^*_+$} is a map $\delta : \R \to \rr^*_+$ satisfying $\delta(x,y)\delta(y,z) = \delta(x,z)$ for all $(x,y),(y,z) \in \R$. 
\end{definition}

\noindent The map $\delta$ is  known as the {\it Radon-Nikodym
  cocycle} of $(\R,T,\nu)$.
  
\begin{definition}
 Given a cocycle $\delta : \R \to \rr^*_+$, the measure $|\cdot|_y$ on
 $\R[y]$ is given by $|z|_y=\delta(z,y)$ for all $z \in \R[y]$. Then,
 for a finite subset $A\subset \R[y]$,
$$|A|_y=\sum_{z\in A} \delta(z,y).$$
\end{definition}

\subsection{Discrete averaging sequences}
\label{sectiondiscretaharmonic} 

We are interested in giving a sufficient condition to solve the Radon-Nikodym problem in the discrete setting. We will state this condition using the notion of modified averaging sequence (see \cite{kaimcomptes} and \cite{kaimexamples}): 

\begin{definition} \label{definciondetalfolner}
Let $\delta : \R \to\rr^\ast_+$ be a cocycle of $\R$.  Let $\{ A_n \}$
be a sequence of finite subsets of $T$ such that $A_n \subset \R[y_n]$ for each $n \in {\mathbb N}$. We will say that $\{ A_n \}$ is a {\em $\delta$-averaging sequence for $\Gamma$} if
$$
\lim_{n \to \infty} \frac{\poids{\Delta_\gamma A_n}{y_n}}{\poids{A_n}{y_n}} = 0
$$
for all $\gamma \in \Gamma^{(1)}$. An equivalence class $\R[y]$ is {\em $\delta$-F{\o}lner} if $\R[y]$ contains an $\delta$-averaging sequence $\{A_n\}$ such that
$\poids{\partial A_n}{y}/\poids{A_n}{y} \to 0$ 
as $n \to +\infty$.
\end{definition}

\noindent
By choosing a finite generating set for $\Gamma$, we can realize each equivalence class $\R[y]$ as the set of vertices of a graph.  We will write $z \sim w$ for each pair of neighboring vertices $z$ and $w$ joined by an edge in $\R[y]$, and $deg(z)$  the number of neighbors of  $z \in \R[y]$. We will use $\mathcal{D}$ to denote the set of discontinuities of the degree function $deg$. Let $\nu$ be a quasi-invariant measure on $T$, and denote by  $D  : L^\infty(T,\nu) \to  L^\infty(T,\nu)$ the Markov operator defined by  
$$
D f(y) = \frac{1}{deg(y)} \sum_{z \sim y} f(z). 
$$
We denote by $D^\ast$ the dual operator acting on the space of positive Borel measures on $T$, 
and by $\Delta  : L^\infty(T,\nu) \to  L^\infty(T,\nu)$ the Laplace operator defined by
$\Delta f(y) = D f(y) - f(y)$.

\begin{definition} \label{defharm}
A quasi-invariante measure $\nu$ on $T$ is {\em harmonic} or {\em stationary} (for the simple random walk on $\R$) if for every bounded measurable function $f  : T \to \rr$, we have 
$\int \, \Delta f \, d\nu = 0$.
\end{definition}

\begin{proposition}[\cite{paulin}]
For a quasi-invariant measure $\nu$ on $T$ the following are equivalent:
\begin{list}{\labelitemi}{\leftmargin=20pt}
\item[(i)] $\nu$ is harmonic;

\item[(ii)] $D^\ast \nu = \nu$; 

\item[(iii)] the Radon-Nikodym cocycle $\delta : \R \to \rr^\ast_+$ is
harmonic, {\it i.e.} for $\nu$-almost every  $y \in T$ and every $z \in \R[y]$, we have
\end{list}
\begin{equation*} 
\delta(z,y) = \frac{1}{deg(z)} \sum_{w \sim z} \delta(w,y).
\end{equation*} 

\end{proposition}

\begin{theorem} \label{discreteharmonic}
 Let $\R$ be the orbit equivalence relation defined by a finitely
 generated pseudogroup $\Gamma$ acting freely on a compact space $T$. Let
 $\delta : \R \to \rr^\ast_+$ be a continuous cocycle. Then:
\smallskip 

\noindent 
i) any $\delta$-averaging sequence $\{ A_n \}$ gives rise to a positive
Borel  measure $\nu$ on $T$ whose support  is contained in the limit
set of $\{ A_n \}$, which is quasi-invariant and has $\delta$ as Radon-Nikodym cocycle; 
\smallskip 

\noindent 
ii) moreover, 
if $\delta$ is harmonic and
$\nu(\mathcal{D}) = 0$, then $\nu$ is a harmonic measure.  
\end{theorem}

\begin{proof}
We start by constructing a sequence of probability measures $\nu_n$ given by
$\nu_n(B) = \poids{B \cap A_n}{y_n} / \poids{A_n}{y_n}$
for every Borel subset $B$ of $T$. By passing to a subsequence, the sequence $\nu_n$ converges in the weak topology to a positive Borel measure $\nu$ on $T$. First, we will prove that $\nu$ is a quasi-invariant measure having a Radon-Nikodym cocycle equal to $\delta$. For every local transformation $\gamma \in \Gamma$ an every function $f \in C(T)$ with support on the range of $\gamma$, we have
$$
\int \, f(z) \, d(\gamma_\ast \nu)(z)  = \int \, f(\gamma(y)) \,
d\nu(y)  =  \lim_{n \to \infty} \frac{1}{\poids{A_n}{y_n}} \sum_{y \in
  A_n}  f(\gamma(y))\delta(y,y_n)
$$
and 
\begin{eqnarray*}
\int \, f(y) \delta(z,y) \, d\nu(y)  & = & \lim_{n \to \infty}
\frac{1}{\poids{A_n}{y_n}} \sum_{y \in A_n}  f(y)\delta(\gamma(y),y)\delta(y,y_n) \\
 & = & \lim_{n \to \infty} \frac{1}{\poids{A_n}{y_n}} \sum_{y \in A_n}
 f(y)\delta(\gamma(y),y_n)
 \end{eqnarray*}
where $z = \gamma(y)$. Therefore 

\begin{eqnarray*}
0 &  \leq & \left| \int \, f(z) \, d(\gamma_\ast \nu)(z) -  \int \, f(y) \delta(z,y) \, d\nu(y) \,  \right| \\
& \leq & \lim_{n \to \infty} \frac{\quad 1 \quad}{\poids{A_n}{y_n}} \,
\left| \sum_{y \in A_n}  f(\gamma(y))\delta(y,y_n) - 
 f(y)\delta(\gamma(y),y_n) \,  \right|  \\
& \leq & \lim_{n \to \infty}  \norminf{f}  \frac{\poids{\Delta_\gamma A_n}{y_n} }{\poids{A_n}{y_n}}  \ = \ 0
 \end{eqnarray*}
and thus
$$\int \, f(z) \, d(\gamma_\ast \nu)(z)  = \int \, f(y) \delta(z,y) \, d\nu(y),$$
proving the claim. 
\medbreak

We will now prove that if $\delta$ is harmonic and 
$\nu(\mathcal{D}) = 0$, then $\nu$ is a harmonic measure. 
Observe that if  $\nu(\mathcal{D}) = 0$, then $\Delta f$ is continuous $\nu$-almost everywhere and therefore
$$
\int \, \Delta f  \, d\nu = \lim_{n \to \infty} \int \, \Delta f \, d\nu_n
$$
for all $f \in C(T)$. If $\delta$ is harmonic, we have
\begin{eqnarray*}
 \int  \, \Delta f(y) \,  d\nu_n(y)  & = &  \frac{1}{\poids{A_n}{y_n}} 
\sum_{y \in A_n}  \left( \frac{1}{deg(y)} \sum_{z \sim y} f(z) - f(y) \right) \delta(y,y_n) 
 \end{eqnarray*}
\begin{eqnarray*}
& = &  \frac{1}{\poids{A_n}{y_n}} 
\sum_{y \in A_n} \frac{1}{deg(y)}  \sum_{z \sim y}  f(z) \delta(y,y_n) - f(y) \left( \frac{1}{deg(y)} 
\sum_{z \sim y} \delta(z,y_n) \right)  \\
\hspace*{-3em} & = & \frac{1}{\poids{A_n}{y_n}} 
\sum_{y \in A_n}   \frac{1}{deg(y)}  \sum_{z \sim y} f(z) \delta(y,y_n) -  f(y) \delta(z,y_n)  
 \end{eqnarray*}
and then 
\begin{eqnarray*}
\hspace{0.7cm} 0 & \leq & \left| \int \, \Delta f  (y) \, d\nu(y) \right| \\
& \leq &  \lim_{n \to \infty}  \frac{1}{\poids{A_n}{y_n}}  
\,  \left| \sum_{y \in A_n}  \sum_{z \sim y}Ê
 f(z) \delta(y,y_n) -  f(y) \delta(z,y_n) \, \right| \\
& \leq & \lim_{n \to \infty} \norminf{f}   
 \sum_{\gamma \in \Gamma^{(1)}}Ê
\frac{\poids{\Delta_\gamma A_n}{y_n} }{\poids{A_n}{y_n}} \  \leq \ 
 \lim_{n \to \infty} 2 \, \norminf{f}  \,  |  \Gamma^{(1)} | \, 
\frac{ \poids{\partial A_n}{y_n} }{ \poids{A_n}{y_n} } \ = \ 0, 
 \end{eqnarray*}
that is $\nu$ is a harmonic measure. 
\end{proof}

A similar result can be found in \cite{schapira}. In general, the second part of Theorem~\ref{discreteharmonic} remains valid when the Laplace operator 
$\Delta$ preserves continuous functions. 
This is always true when $\mathcal{D} = \emptyset$, as in the following case: 

\begin{corollary} \label{discreteharmoniccor1}
 Let $\R$ be the orbit equivalence relation defined by a group of
 finite type $\Gamma$ acting freely on a compact space $T$.  Let
 $\delta : \R \to \rr^\ast_+$ be a continuous harmonic cocycle. Any
 $\delta$-averaging sequence $\{ A_n \}$ gives rise to a harmonic
 measure $\nu$ on $T$ supported by the limit set of $\{ A_n  \}$. \qed
 \end{corollary}
 
\noindent
Arguing as for usual averaging sequences, we can extend Theorem~\ref{discreteharmonic} to any compactly generated pseudogroup $\Gamma$ acting freely on a locally compact Polish space $T$. 
Moreover, in the $0$-dimensional case, the degree function is again continuous.
This applies in particular to solenoids \cite{benedettigambaudo} and laminations defined by repetitive graphs (introduced in \cite{ghyspano} and studied in \cite{alcaldelozanomacho}, \cite{theseblanc} and \cite{lozanoexample}): 

\begin{corollary} \label{discreteharmoniccor2}
 Let $\R$ be the orbit equivalence relation defined by compactly
 gene\-rated pseudogroup $\Gamma$ acting freely  on a locally compact separable $0$-dimensional space $T$.  Let $\delta : \R \to \rr^\ast_+$ be a continuous harmonic cocycle. Any $\delta$-averaging sequence $\{ A_n \}$ gives rise to a harmonic measure $\nu$ on $T$ supported by the  limit set of $\{ A_n \}$. \qed
 \end{corollary}

In order to extend Theorem~\ref{discreteharmonic} to non-free actions, we can adopt two different
strategies. Let us first recall that the notion of equivalence
relation is enough to describe the transverse structure of a
lamination in the Borel context. More precisely, any Borel or
topological lamination $\F$ induces a Borel equivalence relation $\R$
on a total transversal $T$ (compare to Remark~\ref{etale}) defined by
the action of the holonomy pseudogroup. We refer to the Ph.D. thesis
of M. Berm\'udez \cite{thesebermudez} for the definition of a Borel lamination. 
If $\R$ is a discrete Borel equivalence relation defined by the action
of a Borel pseudogroup $\Gamma$ acting on a compact
space $T$ and if $\delta : \R \to \rr^\ast_+$ is a Borel cocycle, then
the proof of Theorem~\ref{discreteharmonic} remains valid.
In the topological context, Theorem~\ref{discreteharmonic} is not exactly equivalent to the situation above because the transverse holonomy groupoid and the equivalence relation are only Borel isomorphic on the residual set of leaves without holonomy. Another strategy consists in
replacing \'etale equivalence relations with \'etale groupoids, and proving that averaging sequences 
for stationary cocycles define stationary measures on groupoids. Details will be reported elsewhere. 

\setcounter{equation}{0}

\section{Averaging sequences in the continuous setting} 
\label{sectionaveragingcontinuous}

We are interested in stating Theorem \ref{discreteharmonic}
in the continuous setting,  namely for a compact laminated space $(M,\F)$. Instead of working with quasi-invariant measures, we are going to use tangentially smooth measures. These form a larger class than harmonic measures. As previously mentioned, transverse invariant measures for foliations are rather rare, but harmonic measures always exist.  Harmonic measures were introduced by L. Garnett in \cite{garnett}. In Sections \ref{medidas} and \ref{medidasharmonicas} we will 
study these measures and recall some notation. In Section 
\ref{sectionmodular} we will construct a differential foliated 1-form from a given cocycle. Finally, in Section \ref{sectioncontinuous} we will use this foliated form to prove 
the continuous analogue of Theorem \ref{discreteharmonic}.

\subsection{Tangentially smooth measures}
\label{medidas}

Consider now a regular Borel measure $\mu$  on $M$. Using a $C^r$ foliated atlas $\A$, we can give a local decomposition $\mu=\int \lambda_i^yd\nu_i(y)$ on each foliated chart $U_i$, where $\lambda_i^y$ is a measure on the plaque $\varphi_i^{-1}(P_i \times \{y\})$ and $\nu_i$ a measure on $T_i$. 
Here,  in order to define the foliated Laplace operator $\Delta_\F$, we can always assume that $r \geq 3$ up to $C^1$-equivalence of foliated atlases, and we fix a tangentially $C^r$-smooth Riemannian metric $g$ along the leaves of $\F$.

\begin{definition}[\cite{alcalderechtman}]
\label{definiciontangencialmentelisa}
A measure $\mu$ on $M$ is {\em tangentially smooth} if for every $i\in I$ and $\nu_i$-almost every $y\in T_i$, the measures $\lambda_i^y$ are absolutely continuous with respect to the Riemannian volume $dvol$ restricted to the plaque passing through $y$, and the density functions
$h_i(x,y)=d\lambda_i^y / dvol(x,y)$
are smooth functions of class $C^{r-1}$ on the plaques.
\end{definition}
 
Observe that the local decomposition of $\mu$ is not necessarily unique. Let
 $\mu|_{U_i}=\int \lambda_i^yd\nu_i(y)=\int \bar{\lambda}_i^yd\bar{\nu}_i(y)$
 be two decompositions. Then we obtain 
 $$
\int_{T_i} \int_{P_i\times\{y\}} \, h_i(x,y) \, dvol(x,y) \, d\nu_i(y)=\int_{T_i }\int_{P_i\times \{y\}}
\, \bar{h}_i(x,y) \, dvol(x,y) \, d\bar{\nu}_i(y),
$$
and we can consider the Radon-Nikodym derivative $\delta_i(y)= d\nu_i / d\bar{\nu}_i(y)$ such that
$\bar{h}_i(x,y) =\delta_i(y)h_i(x,y)$.
This situation arises naturally in the intersection of two foliated charts $U_i$ and $U_j$.  Indeed, if $U_i \cap U_j \neq \emptyset$, 
 we have that
$\mu|_{U_i\cap U_j}=\int \lambda_i^yd\nu_i(y)=\int \lambda_j^yd\nu_j(y).$
Thus, as before, we deduce that
\begin{equation} \label{eqcocycle1}
\delta_{ij}(y)  = d\nu_i / d((\gamma_{ji})_*\nu_j) (y) = 
 \frac{h_j(\varphi_{ij}^y(x), \gamma_{ij}(y))}{h_i(x,y)}.
\end{equation}
Then the functions $h_i$ verify that $\log h_j-\log h_i=\log \delta_{ij}$ on $U_i\cap U_j$. Since $\delta_{ij}$ is a function on $T_i$, we have that
 $ d_\mathcal{F}\log h_i=d_\mathcal{F}\log h_j$.
Then $\eta= d_\mathcal{F}\log h_i$ is a well-defined foliated $1$-form of class $C^{r-2}$ along the leaves, which makes possible to estimate the transverse measure distortion under the holonomy. 

\begin{definition}
 \label{definicionformamodular}
 The foliated $1$-form $\eta$ is the {\em modular form} of $ \mu$.
 \end{definition}

\subsection{Harmonic measures} \label{medidasharmonicas}
 We start by recalling the definition given by L. Garnett in \cite{garnett}:

\begin{definition} 
We will say that $\mu$ is {\em harmonic} if $\int \Delta_{\mathcal{F}} f d\mu = 0$ for every conti\-nuous tangentially $C^{r-1}$-smooth function $f : M \to \rr$. 
\end{definition} 
 
\noindent
According to Theorem 1 of \cite{garnett}, any harmonic measure is an
example of tangentially smooth measure since the densities $h_i$ are
positive harmonic functions of class $C^{r-1}$ on the plaques.  In
particular, any transverse invariant measure combined with the
Riemannian volume on the leaves gives a harmonic measure which is
called {\em completely invariant}. A harmonic measure $\mu$ is
completely invariant if and only if $\eta=0$ (we refer to corollary
5.5 of A. Candel's paper \cite{candel}).  In the general harmonic
case, the following proposition states some properties of the modular form. This proposition is a refined version of Lemma 4.19 on page 116 of the Ph.D. thesis of B. Deroin \cite{thesederoin}. 

\begin{proposition}[\cite{thesederoin}] \label{harmonicdensity}
If $\mu$ is a harmonic measure, then $\eta$ is a bounded foliated $1$-form which admits a uniformly tangentially Lipschitz primitive $\log h$ on the residual set of leaves without holonomy.
\end{proposition}

\begin{proof} Let $\mathcal{A}=\{(U_i,\phi_i)\}_{i\in I}$ be a good $C^r$ foliated atlas
of $(M,\F)$, and $h_i$ the local density functions of $\mu$. Let us first observe that since the functions $h_i$ coincide on the intersections of the plaques modulo multiplication by a constant, they define a primitive of the induced $1$-form on the holonomy covering of each leaf $L$. If $\mathcal{F}$ has no essential holonomy, the functions $\log h_i$ can be glued together to obtain a measurable global primitive $\log h$ of $\eta$. In general, the modular form $\eta$ admits a continuous primitive $\log h$ on the residual set of leaves without holonomy. Now, let us assume
that $\mathcal{A}$ is a refinement of a good atlas 
$\mathcal{A}^\prime = \{(U^\prime_i,\phi^\prime_i)\}_{i\in I}$, 
and $h_i^\prime$ are the corresponding local densities. 
Thus, every plaque of $U_i$ is relatively compact in a plaque of $U_i^\prime$. 
In fact, using a vertical reparametrization, we can suppose that 
$\phi_i^{-1}(P_i \times \{y\})  \subset (\phi_i^\prime)^{-1}(P_i^\prime \times \{y\})$ for every $y\in T_i$. There exists a relatively compact open set $V \subset P_i^\prime$ such that
$ \phi_i^{-1}(P_i \times \{y\}) \subset (\phi_i^\prime)^{-1}(V \times \{y\})$
for every $y\in T_i$. Since $h_i$ is harmonic, the Harnack inequality implies the existence of a constant $C_i>0$ such that
\begin{equation} \label{Harnack}
 \frac{1}{C_i} \leq \frac{h_i(x,y)}{h_i(x_0,y)} \leq C_i ,
\end{equation}
for all $x, x_0\in P_i$ and for all $y\in T_i$. Since the atlases $\mathcal{A}$ and $\mathcal{A}^\prime$ are finite, the primitive $\log h$ is uniformly Lipschitz in the tangential coordinate $x$.
\end{proof} 

\subsection{Modular form associated to a cocycle}\label{sectionmodular}

We will now describe how to cons\-truct a modular form $\eta \in \Omega^1(\F)$ from a Borel or 
continuous cocycle $\delta : \R \to \rr^\ast_+$. For simplicity, 
$\R$ is endowed here with the natural Borel or topological structure induced by the structure of Borel or topological groupoid on the transverse holonomy groupoid $G$ formed by the germs 
$< \! \gamma \! >_y$ of the elements $\gamma$ of $\Gamma$ at the points $y$ of their domains,  see \cite{mooreschochet}.
The natural projection $(\beta,\alpha) : < \! \gamma \! >_y \in G \mapsto (y,\gamma(y)) \in \R$ becomes an isomorphism of Borel or topological groupoids in restriction to the residual set of leaves without holonomy. Equivalently, we can consider a Borel or continuous cocycle $\delta : G \to \rr^\ast_+$ projectable on $\R$.
\medskip 

We start by considering tangentially $C^r$-smooth Borel or continuous functions 
$c_{ki} : U_i \cap U_k \to \rr$ given  by
$$c_{ki} (\varphi_k^{-1} (x,y)) = log \, \delta_{ki}(y)$$
where $\delta_{ki}(y) = \delta(y,\gamma_{ki}(y))$ for all $(x,y) \in P_k \times T_k$.  By choosing a tangentially $C^r$-smooth partition of unity $\{ \rho_i \}_{i=1}^m$ subordinated to the foliated atlas $\A$, we can glue the functions $c_{ki}$ obtaining  tangentially $C^r$-smooth Borel or continuous functions 
$c_i : U_i \to \rr$
given by
$$c_i = \sum_{k=1}^m \rho_k c_{ki}.$$
The cocycle condition implies that $c_{ij} = c_{kj} - c_{ki}$,
so that
$$
c_j - c_i =  \sum_{k=1}^m \rho_k c_{kj} - \sum_{k=1}^m \rho_k c_{ki} = \left(\sum_{k=1}^m \rho_k\right) c_{ij} =  c_{ij}.
$$
Hence, for each $i = 1, \dots, m$ we can define a tangentially $C^{r-1}$-smooth Borel or continuous foliated $1$-form 
$$\eta_i = \sum_{k=1}^m (d_\F \rho_k) \, c_{ki}$$
on $U_i$. Each local $1$-form $\eta_i$ is exact
$$
\eta_i = \sum_{k=1}^m (d_\F \rho_k) \, c_{ki} = d_\F c_i = d_\F \log h_i
$$
where $h_i = e^{c_i} : U_i \to \rr^\ast_+$ is a Borel or continuous function of class $C^r$ along the leaves.

\begin{proposition} \label{formamodular}
There is a well defined Borel or continuous closed foliated $1$-form $\eta \in \Omega^1(\F)$ such that
$\eta |_{U_i} = \eta_i$.
\end{proposition}

\begin{proof}
For each pair $i,j \in \{1, \dots, m\}$, we have that: 
$$
\eta_j - \eta_i = \sum_{k=1}^m (d_\F \rho_k) \,  c_{kj}  -  \sum_{k=1}^m (d_\F \rho_k) \, c_{ki} = 
\left( \sum_{k=1}^m d_\F \rho_k\right) \, c_{ij} = 0
$$
on $U_i \cap U_j$. Then the $1$-form $\eta$ is well defined,  Borel or continuous, and closed.
\end{proof}

\begin{definition}
The foliated 1-form $\eta$ is the modular form of $\delta$.
\end{definition}

\begin{remarks}  \label{remmod2}
\begin{list}{\labelitemi}{\leftmargin=0pt}

\item[\, (i)] The modular form $\eta$ depends on the choice of the partition of unity, but its 
cohomology class does not depend. 

\item[\, (ii)] As for harmonic measures, the modular form $\eta$ of a
  Borel or continuous cocycle $\delta$ admits a Borel or continuous
  primitive $\log h$ on the residual set of leaves without
  holonomy. Thus, assuming that $\F$ has no holonomy (or passing to
  the holonomy covers of the leaves), we may find a global Borel or
  continuous primitive on $M$ (respectively, a Borel or continuous primitive on the holonomy groupoid $Hol(\F)$), see \cite{alcalderechtman}.
\end{list}
\end{remarks}

\subsection{Continuous averaging sequences} \label{sectioncontinuous}

In the present setting, we can reformulate the {\em Radon-Nikodym problem}  as the problem 
of determining tangentially smooth measures $\mu$ on $M$ which admit $\eta$ as their modular form. 
The aim of this section is to establish Theorem \ref{discreteharmonic}
for laminations. First, we need a continuous analogue of Definition \ref{definciondetalfolner}. 
Consider a $d$-dimensional lamination $\F$ of class $C^r$ on a compact space $M$, endowed with a tangentially $C^r$-smooth Riemannian metric  $g$, and a continuous cocycle
$\delta :  \R \to \rr_+^\ast$. The modular form $\eta$ admits a continuous tangentially $C^r$-smooth primitive $\log h$ on the residual set of leaves without holonomy. On each leaf 
without holonomy $L_y$ passing through $y \in T$, we can multiply $g$ by the normalized density function $h/h(y)$ in order to obtain a {\it modified metric} $(h/h(y))g$. 

\begin{definition}
\label{definicionfolner}
Let $\{ V_n \}$ be a sequence of compact domains with boundary contained in a sequence of leaves  without holonomy $L_{y_n}$. We will say that $\{ V_n \}$ is a {\em $\eta$-averaging sequence for $\F$} if
$$
\lim_{n \to \infty} \frac{\mbox{area}_\eta(\partial V_n)}{\mbox{vol}_\eta(V_n)} = 0
$$
where $\mbox{area}_\eta$ denotes the $(d-1)$-volume and $\mbox{vol}_\eta$ the $d$-volume with respect to the modified metric along $L_{y_n}$. A leaf $L_y$ is {\em $\eta$-F{\o}lner} if it contains an $\eta$-averaging sequence $\{V_n\}$ such that
$\mbox{area}_\eta(\partial V_n) / \mbox{vol}_\eta(V_n) \to 0$ as $n \to +\infty$.
\end{definition}

\begin{remarks} \label{harnackremark}

\begin{list}{\labelitemi}{\leftmargin=0pt}

\item[\, (i)]  The isoperimetric ratio $\mbox{area}_\eta(\partial V_n) / \mbox{vol}_\eta(V_n)$ does not depend on the  choice of $y$ nor $h$ in the second definition. 
This justifies the notation, which is slightly different from the one used in \cite{alcalderechtman}. 

\item[\, (ii) ]  When $\mu$ is a completely invariant harmonic measure, the normalized density function is equal to $1$ and thus the modified volume and the Riemannian volume coincide. Hence, we recover  the common definition of averaging sequence. 

\item[\, (iii)] 
For harmonic measures, Harnack's inequalities~(\ref{Harnack}) imply that
the modified volume of the plaques and the modified area of their boundaries remain uniformly bounded. 
\end{list}
\end{remarks}

\begin{theorem} \label{thharmonic}
Let $(M,\F)$ be a $C^r$ lamination of a compact space $M$, $1 \leq r \leq \infty$, and let 
$\R$ be the equivalence relation induced by $\F$ on a total transversal $T$. Consider a continuous cocycle $\delta :  \R \to \rr_+^\ast$, and let $\eta$ be the modular form of  $\delta$.  Assume that $\F$ admits a foliated atlas such that the modified volume of the plaques is bounded. Then:
\smallskip 

\noindent 
i) any $\eta$-averaging sequence $\{V_n\}$ for $\F$ gives rise to a tangentially smooth measure $\mu$ whose support  is contained in the limit set of $\{V_n\}$ and whose modular form is equal to $\eta$; 
\smallskip 

\noindent 
ii) moreover, 
if $\eta$ has a primitive $\log h$ such that $h$ is a harmonic function, then $\mu$ is a harmonic measure. 
\end{theorem}

\begin{proof}
As in the discrete case, we will start by constructing a sequence of foliated $d$-currents
$$
\xi_n(\alpha) = \frac{1}{\mbox{vol}_\eta(V_n )} \int_{V_n}
\frac{h}{h(y_n)} \alpha,
$$
where $\alpha$ is a foliated $d$-form.
By passing to a subsequence, the sequence $\xi_n$ converges to a foliated $d$-current $\xi$. 
Let $\mu$ be the measure on $M$ associated with the current $\xi$. For every function $f \in C(T)$, we have 
$\int \, f \, d\mu = \xi (f \omega)$ 
where $\omega = dvol$ is the volume form along the leaves. 
\medbreak

Now, we will prove that $\mu$ is a tangentially smooth measure with
modular form $\eta$.  Consider a good $C^r$ foliated atlas
$\mathcal{A}=\{(U_i,\phi_i)\}_{i\in I}$ obtained by refinement from a
given good atlas, and whose plaques have bounded modified volume. 
As we mentioned before, up to $C^1$-equivalence, we can assume now
that $r \geq 3$. Since the modified volume of the plaques of
$\mathcal{A}$ and the modified area of their boundaries remain
bounded,  the traces $A_n = V_n \cap T$ of the domains $V_n$ on the
total transversal $T$ form a $\delta$-averaging sequence, as in
Definition \ref{definciondetalfolner}. In fact, since $V_n$ is covered by the plaques $P_y$ of $\mathcal{A}$ centered at the points $y$ of $A_n$, we have that: 
$$
\mbox{vol}_\eta(V_n) = \int_{V_n} \omega_\eta \leq \sum_{y \in A_n} \int_{P_y} \omega_\eta = 
\sum_{y \in A_n} \left( \int_{P_y} \, \frac{h(x,y)}{h(0,y)} dvol(x,y) \, \right) \delta(y,y_n)
$$
where $\omega_\eta$ is the modified volume form along the leaves and $h(x,y)$ denotes the density function  restricted to a  foliated chart $U_y$ containing the plaque $P_y$.  Then there is a constant 
$C > 0$ such that 
$\mbox{vol}_\eta(V_n) \leq C \poids{A_n}{y_n}$
Actually, we can choose $C > 0$ such that 
$\frac{1}{C} \leq \mbox{vol}_\eta(V_n) / \poids{A_n}{y_n} \leq C$.
Thus, by passing to a subsequence, we may assume that the ratio 
$\mbox{vol}_\eta(V_n) /\poids{A_n}{y_n}$ converges to a constant $c > 0$. 
Now, as stated in the proof of Theorem~\ref{discreteharmonic}, we may also assume that the sequence of measures $\nu_n(B) = \poids{B \cap A_n}{y_n} / \poids{A_n}{y_n}$
converge to a quasi-invariant measure $\nu$ on $T$ whose  Radon-Nikodym derivative is equal to 
$\delta$. Combined with the modified Riemannian volume along the leaves, 
this transverse measure gives us a tangentially smooth measure $\mu'$ on $M$. 
Thus, for every function 
$f \in C(M)$ with support 
 in $U_i$,  we have
$$
\int \, f \, d\mu'  =
 \int_{T_i}  \int_{P_i \times \{y\}}  \, f(x,y)  \, \frac{h_i(x,y)}{h_i(0,y)}  \, dvol(x,y)  \, d\nu(y).
 $$
Then
\begin{eqnarray} 
\int \, f \, d\mu' 
& = & \lim_{n \to +\infty} \frac{1}{\poids{A_n}{y_n}} \sum_{y \in V_n \cap T_i} 
\left( \int_{P_i \times \{y\}} \, f(x,y) \, \frac{h_i(x,y)}{h_i(0,y)}  \, \, dvol(x,y) \right) \, \delta(y,y_n) \nonumber \\
& = & \lim_{n \to +\infty} \frac{1}{\poids{A_n}{y_n}} \sum_{y \in V_n \cap T_i} \int_{P_i \times \{y\}} \, f \, \omega_\eta  \label{eq1}.
\end{eqnarray} 
On the other hand, by definition, we have
\begin{eqnarray}
\int \, f \, d\mu \quad  = \quad  \xi (f \omega) \nonumber 
& = &  \lim_{n \to +\infty} \frac{1}{\mbox{vol}_\eta(V_n )} \int_{V_n} \, f \omega_\eta  \nonumber \\ 
& = & \lim_{n \to +\infty} \frac{1}{\mbox{vol}_\eta(V_n )}  \sum_{y \in V_n \cap T_i}  \int_{P_i \times \{y\}} \, f \, \omega_\eta \label{eq2}
\end{eqnarray} 
Comparing identities~(\ref{eq1})~and~(\ref{eq2}), we deduce that $\mu = \frac{1}{c} \, \mu'$ is a tangentially smooth measure with 
modular form $\eta$.

\medbreak
To conclude, we will prove that $\mu$ is harmonic when $h$ is harmonic. 
We will start by denoting $h_n = h/h(y_n)$ the normalized density functions on the leaves $L_{y_n}$.
Since the Laplace operator $\Delta_\F$ preserves continuous functions, we have that
$$
\int \, \Delta_\F f \, d\mu = \lim_{n \to \infty} \frac{1}{\mbox{vol}_h(V_n )} \int_{V_n} (\Delta_\F f) \,  h_n \,  \omega,
$$
for all $f \in C(T)$. Green's formula implies that
$$
\int_{V_n}  \, (\Delta_\F f) \, h_n \, \omega  =  
\int_{V_n}  \, ( (\Delta_\F f ) \, h_n  - f \, (\Delta_\F h_n) \, \omega  = 
 \int_{\partial V_n}  \, h_n  \,\iota_{grad(f)} \omega -  f \,\iota_{grad( h_n)} \omega. 
$$
Since $h_n$ is harmonic, we have
$$
\int_{\partial V_n}  \,\iota_{grad(h_n)} \omega = \int_{ V_n}  \, \mbox{div}(grad(h_n)) \, \omega =  \int_{ V_n} \, (\Delta_\F h_n) \, \omega = 0
$$
and then 
$$
0 \leq \left| \int_{\partial V_n} f\iota_{grad(h_n)} \omega \, \right| 
\leq  \norminf{f}  \int_{\partial V_n}  \,\iota_{grad(h_n)} \omega = 0
$$
for all $n \in \mathbb{N}$. 
On the other hand, since $f$ is bounded, there exists a constant $k > 0$ depending only on $f$ such that we have
$$
0 \leq \left| \, \frac{1}{\mbox{vol}_h(V_n )}  \int_{\partial V_n}  \, h_n \,\iota_{grad(f)}  \omega \, \right| 
\leq  \lim_{n  \to \infty} k \frac{\mbox{area}_\eta (\partial V_n)}{\mbox{vol}_\eta(V_n )} = 0
$$
and therefore 
$$
\hspace{2.2cm} \int \, \Delta_\F f \, d\mu = \lim_{n \to \infty} \frac{1}{\mbox{vol}_h(V_n )} \int_{V_n} (\Delta_\F f) \,  h_n \,  \omega = 0, 
$$
that is $\mu$ is a harmonic measure. 
\end{proof}

\begin{remarks} 
\begin{list}{\labelitemi}{\leftmargin=0pt}

\item[\, (i)] If $\delta :  \R \to \rr_+^\ast$ is a  Borel cocycle with modular form $\eta$, Theo\-rem~\ref{thharmonic}  remains also valid. So any $\eta$-averaging sequence for $\F$ gives rise to a tangentially smooth measure $\mu$ that is harmonic when $\eta$ admits a primitive $\log h$ such that $h$ is a harmonic  function.

\item[\, (ii)]  According to Remark~ \ref{remmod2}.(ii), the notion of $\eta$-F{\o}lner  may be applied to the holonomy covers of the leaves of $\mathcal F$. Thus, it suffices to replace $\F$ with the lifted lamination in the holonomy groupoid $Hol(\F)$, in order to globalize the previous result. As in the discrete setting, details will be precised elsewhere.
\end{list}
\end{remarks}

\setcounter{equation}{0}

\section{Examples}
\label{Sex}

\subsection{Discrete averaging sequences for amenable non F\o lner actions.}
There are amenable actions of non amenable discrete groups whose
orbits contain avera\-ging sequences \cite{kaimexamples}. For example, let $\partial \Gamma$ be the space of ends of the free group $\Gamma$ with two generators $\alpha$ and $\beta$ whose elements are infinite words $x = \gamma_1\gamma_2\dots$ with 
letters $\gamma_n$ in $\Phi = \{ \alpha^{\pm 1},\beta^{\pm 1} \}$. If $\nu$ denotes the equidistributed probability measure on $\partial \Gamma$ (such that all cylinders consisting of infinite words with fixed first $n$ letters have the same measure), then $\Gamma$ acts essentially freely on $\partial \Gamma$ by sending 
each generator $\gamma$ and each infinite word $x = \gamma_1\gamma_2\dots $ to 
$\gamma.x = \gamma \gamma_1\gamma_2\dots $. 
Since this action is amenable, according to Theorem~2 of \cite{kaimcomptes} (see also Proposition 4.1 of \cite{alcalderechtman}), we know that $\nu$-almost every orbit is $\delta$-F\o lner (where $\delta$ is the Radon-Nikodym derivative of $\nu$). We will recall here an explicit construction by V. A. Kaimanovich in \cite{kaimexamples}.
\medbreak

For each $x\in \partial \Gamma$, let $b_x : \Gamma \to \rr$ be the {\em Busemann function} defined by 
$$b_x(\gamma) = \lim_{n \to +\infty} \left( d_\Gamma(\gamma,x_{[n]}) - d_\Gamma(1,x_{[n]}) \right)$$
where $d_\Gamma$ is the Cayley graph metric, $x_{[n]} $ is the word consisting of first $n$ letters of $x$ and $1$ is the identity element. The level sets 
$H_k(x) = \{ \, \gamma \in \Gamma \, / \, b_x(\gamma) = k \, \}$
are the {\em horospheres} centered at $x$. 
The Radon-Nikodym derivative of $\nu$ is given by 
$$
\delta(\gamma^{-1}.x,x) = \frac{d \gamma.\nu}{d\nu }(x) = 3^{-b_x(\gamma)}
$$
where $\gamma.\nu$ is the translation of $\nu$ by $\gamma$. Since $\poids{\cdot}{x} = \delta(\cdot,x)$ is a harmonic measure on $\Gamma.x$, $\nu$ is also a harmonic measure. In fact, as stated in Theorem~17.4 of \cite{kaimpoisson}, 
$\nu$ is the unique harmonic probability measure on $\partial \Gamma$. 
\medskip

Let $A_n^x$ be the set of all points $\gamma^{-1}.x$ in $\Gamma.x$
such that $0 \leq b_x(\gamma) = d_\Gamma(1,\gamma) \leq n$. Since 
$\poids{A_n^x \cap H_k(x)}{x} =  \sum_{b_x(\gamma) = d_\Gamma(1,\gamma) = k} \delta(\gamma^{-1}.x,x) = 3^k \frac{1}{3^k} = 1$
for all $0 \leq k \leq n$, we have that $\poids{A_n}{x} = n+1$. But 
$\partial A_n^x = \{1\} \cup (A_n^x \cap H_n(x))$
and so $\poids{\partial A_n^x}{x} = 2$. The $\delta$-averaging
sequence $\{ A_n^x \}$ defines a harmonic measure (which is equal to
$\nu$ up to multiplication by a constant). 

\subsection{Averaging sequences for hyperbolic surfaces.}  The geodesic and horocycle flows 
are classical examples of flows on the unitary tangent bundle of a compact hyperbolic surface. 
They are given by the right actions of the diagonal subgroup 
$$
D = \left\{ \ \matrice{e^{t/2}}{0}{0}{e^{-t/2}}  \Big| \ t \in \rr \ \right\}
$$
and the unipotent subgroup 
$$
H^+ = \left\{ \ \matrice{1}{s}{0}{1}  \Big|  \ s \in \rr \ \right\}
$$
of $G = PSL(2,\rr)$ on the quotient $\Gamma \backslash G$ by the left action of a uniform lattice 
$\Gamma$. If $\hh$ denotes the hyperbolic plane, we can identify
$\Gamma \backslash G$ with the unitary tangent bundle of the compact hyperbolic surface $\Gamma \backslash \hh$. The right action of the normalizer $A$
of $H^+$ in $PSL(2,\R)$ defines a foliation $\F$ by Riemann surfaces on $\Gamma \backslash G$. 
Since $A$ is an amenable group, $\F$ is an amenable non F\o lner foliation. Moreover, there is an $A$-invariant measure $\mu$ on $\Gamma \backslash G$. In \cite{garnett},  L. Garnett proved that $\mu$ is a harmonic measure by describing its density function on a foliated chart. 
\medbreak

We can identify $G / A$ with 
the boundary $\partial \hh$ by sending each coset of $A$ in $G$ to the center of the horocycle 
defined by the corresponding coset of $H^+$ in $G$.   
For each point $z \in \hh$, there is a unique probability measure $\nu_z$ on $\partial \hh$ which is invariant by the action of all isometries of $\hh$ fixing $z$. This measure is the image
of the normalized Lebesgue measure on the circle of the tangent plane at $z$ under the exponential map, and  is called the {\em visual measure} at $z$. According to Proposition 2 of \cite{garnett}, the normalized density function is given by
$d\nu_{z} / d\nu_{z_0}(x)$
where $z,z_0 \in \hh$ and $x \in \partial \hh$. In particular, for $x = \infty$, we have that 
$$
\quad  \quad  \quad  \quad 
\frac{d\nu_{z}}{d\nu_{z_0}}(\infty)  = \frac{y}{y_0}
$$
where $z = x + i y$ and $z_0 = x_0 + i y_0$. In the leaf passing through $x = \infty$, the sequence 

\noindent
$V_n^\infty  = \{ \, z  \in \hh \, | \, -1 \leq x \leq 1 \, , \, e^{-n} \leq y \leq 1\, \}$
becomes a $\eta$-averaging sequence (where 
$\eta$ is the modular form of $\mu$). Indeed, 
on the one hand,  we have that
$$
\mbox{area}_\eta (V_n^\infty) = \int_{V_n^\infty} \, \frac{d\nu_{z}}{d\nu_i}(\infty) \,  dvol(z) = 
\int_{V_n^\infty} \, y \, \frac{dx \wedge dy}{y^2} = \int_1^1 dx \int_{e^{-n}}^1 \frac{dy}{y} = 2n. 
$$
On the other hand, the modified length of a smooth curve $\sigma(t) = x(t) + i y(t)$ (with 
$0 \leq t \leq l$) is given by
$\mbox{length}_\eta(\sigma) = \int_0^l \sqrt{x'(t)^2 + y'(t)^2} dt$,
and so we have that
$$
\mbox{length}_\eta (\partial V_n^\infty) = 2 (2+(1-e^{Ðn})) \leq 6.
$$
As before, this $\eta$-averaging sequence defines a harmonic measure (which is equal to $\mu$ up to multiplication by a constant). In fact, all leaves are $\eta$-F\o lner since for each point $x \in \partial \hh$ obtained as the image of $\infty$ under $g \in G$, the sets $V_n^x  =  g(V_n^\infty)$ form a $\eta $-averaging sequence in the leaf passing through $x$. 

\subsection{Averaging sequences for torus bundles over the circle.} To
conclude, we will  present other examples of foliations on homogeneous spaces studied by  \'E. Ghys and V. Sergiescu in \cite{ghyssergiescu}. Each matrix $A \in SL(2,\zz)$ with $| tr(A) | > 2$ defines a natural representation $\varphi : \zz \to Aut(\zz^2)$ which extends to a representation 
$\Phi : \rr \to Aut(\rr^2)$ given by $\Phi(t) = A^t$. If $\lambda > 1$ and 
$\lambda^{-1} < 1$ are the eigenvalues of $A$, then $\Phi$ is conjugated to the representation 
$\Phi_0$ given by 
$$
\Phi_0(t) = \matrice{\lambda^t}{0}{0}{\lambda^{-t}}.$$
Let $T^3_A$ be the homogeneous space obtained as the quotient of the Lie group 
$G = \rr^2 \rtimes_\Phi \rr$ with group law
$(x,y,t).(x',y',t') = ((x,y)+A^t(x',y'),t+t')$
by the uniform lattice $\Gamma = \zz^2 \rtimes_\varphi \zz$ with a similar law. 
Observe that $G$ is isomorphic to the solvable group $Sol^3 = \rr^2 \rtimes_{\Phi_0} \rr$ with 
group law 
$(x,y,t).(x',y',t') = (x+\lambda^tx',y+\lambda^{-t}y',t+t')$
(where $x$ and $y$ are the first and second coordinate with respect to the eigenbasis) 
and $T^3_A$ is diffeomorphic to the quotient of $Sol^3$ by a uniform lattice $\Gamma_0$. 
The right action of the image $A$ of the monomorphism
$$(a,b) \in \rr \rtimes \rr^\ast_+ \mapsto \left(a,0,\frac{\log b}{\log \lambda}\right) \in Sol^3 $$
defines a foliation $\F$ on $T^3_A$. 
The Lebesgue measure on $T^3_A$ defined by the volume form
$\Omega = dx \wedge dy \wedge dt$ is a tangentially smooth measure. Since 
the Riemannian volume along the right orbits is given by 
$$
\frac{da \wedge db}{b^2} = (\log \lambda) \,  \lambda^{-t} dx \wedge dt
$$
the density function is equal to $\frac{\lambda^t}{\log \lambda}$. In the orbit of the identity element, the sequence $V_n  = \{ \, (a,b)  \in A \, / \, -1 \leq a \leq 1 \, , \, e^{- n \log \lambda} \leq b \leq 1 \, \}$
becomes a $\eta$-averaging sequence (where $\eta$ is the modular form of $\mu$). Indeed, 
on one hand,  we have that
$$
\mbox{area}_\eta (V_n) = \int_{V_n} \, \frac{1}{\log \lambda} \lambda^t  (\log \lambda) \lambda^{-t} dx \wedge dt  = \int_1^1 dx \int_{-n}^0 dt = 2n. 
$$
On the other hand, the modified length of a smooth curve $\sigma(t) = (a(t),b(t)$ (with 
$0 \leq t \leq L$) is given by
$\mbox{length}_\eta(\sigma) = \int_0^L \sqrt{a'(t)^2 + b'(t)^2} dt$,
and so we have that
$$
\mbox{length}_\eta (\partial V_n) = 2 (2+(1-e^{Ðn \log \lambda})) \leq 6.
$$
By replacing the orbit corresponding to $y=0$ with another orbit, it is easy to see that all leaves are $\eta$-F\o lner. As in the previous example, all $\eta$-averaging sequen\-ces define (up to multiplication by a constant) the same harmonic measure: the Lebesgue measure.

\setcounter{equation}{0}

\section{Final comments} \label{Sfc}

\subsection{Discrete and continuous averaging sequences} Comparing the discrete and continuous settings, a natural question arises: what is the relation between $\delta$-averaging and $\eta$-averaging sequences? Let us first notice that repeating the same argument as in the classical case (see Theorem 4.1 of \cite{kanai}), the boundedness condition derived from Harnack's inequalities in Remark~\ref{harnackremark}.(iii) implies that {\em the leaf $L_y$ is $\eta$-F{\o}lner if and only if the equivalence class $\R[y]$ is $\delta$-F{\o}lner}. But then, what is the relation between the harmonic measures defined by $\delta$-averaging and $\eta$-averaging sequences? In this case, the answer is more subtle, and we have to use an important result of R. Lyons and D. Sullivan  \cite{lyonssullivan}, completed later by V. A. Kaimanovich \cite{kaimdiscrete} and independently by W. Ballman and F. Ledrappier \cite{ballmannledrappier}, about the discretization of harmonic functions on Riemannian manifolds. First, according to Theorem 6 of \cite{lyonssullivan},  if $\mu$ is a harmonic measure, then the transverse measure $\nu$ (well defined up to equivalence) is $\pi$-harmonic where $\pi$ is a transition kernel defining a random walk on $\R$ different from the simple random walk considered in Definition~\ref{defharm}. Reciprocally, assuming that $T$ admits a relatively compact neighborhood which meets almost every leaf in a recurrent set, the Main Theorem  of \cite{ballmannledrappier} implies that $\mu$ is harmonic if $\nu$ is $\pi$-harmonic.
\medskip

\subsection{Amenability} It is not casual that all examples in
Section~\ref{Sex} are amenable: according to a resul by V. A. Kaimanovich \cite{kaimcomptes}, amenable foliations admit always averaging sequences. In fact, if $\F$ is an amenable foliation with respect to a tangentially smooth measure $\mu$, then $\F$ is $\eta$- F\o lner, i.e. $\mu$-almost every leaf is $\eta$-F\o lner, see Proposition 4.3 of \cite{alcalderechtman}. This paper can be viewed as a sequel of \cite{alcalderechtman} where we proved that minimal $\eta$-F\o lner foliations are $\mu$-amenable (assuming that the modified volume of the plaques is bounded). To complete the series, we have to prove that any foliation is amenable with respect to a tangentially smooth measure $\mu$ constructed from an averaging sequence using Theorem~\ref{thharmonic}.

\end{document}